\documentclass{article}

\usepackage{amsmath,amssymb,amsthm}

\theoremstyle{definition} \newtheorem{definition}{Definition}
\theoremstyle{definition} \newtheorem{hypothesis}{Hypothesis}
\theoremstyle{definition} \newtheorem*{remark}{Remark}
\theoremstyle{plain}    \newtheorem{theorem}{Theorem}
                          
                          \newtheorem{lemma}[theorem]{Lemma}

\title{Well-posedness for Stochastic Evolution Equations with Monotone Non-linearity and Multiplicative Poisson Noise in $L^p$}

\author{Erfan Salavati\footnote{This research was supported by a grant from IPM.}\\
Faculty of Mathematics and Computer Science\\
Amirkabir University of Technology (Tehran Polytechnic)\\
School of Mathematics,\\
Institute for Research in Fundamental Sciences (IPM),\\
P.O. Box: 19395-5746, Tehran, Iran\\
Bijan Z. Zangeneh\\
Department of Mathematical Sciences\\
Sharif University of Technology\\
Tehran, Iran}
\date{}

\begin{document}
\maketitle

\begin{abstract}
Semilinear stochastic evolution equations with L\'evy noise and monotone nonlinear drift are considered. The existence and uniqueness of the mild solutions in $L^p$ for these equations is proved and a sufficient condition for exponential asymptotic stability of the solutions is derived. The main tool in our study is an It\^o type inequality for the $p$th power of stochastic convolution integrals in Hilbert spaces.
\end{abstract}

\vspace{2mm}
\vspace{2mm}


\section{Introduction}\label{section: introduction}

In this article we are concerned with the $L^p$ well-posedness (existence, uniqueness and continuous dependence on initial conditions) of the equation,
\begin{equation}\label{Poisson_Noise}
dX_t=AX_t dt+f(t,X_t) dt + \int_E k(t,\xi,X_{t-}) \tilde{N}(dt,d\xi),
\end{equation}
on a Hilbert space $H$, where $A$ in the generator of a $C_0$-semigroup of operators on $H$, $f$ is a non-linear and dissipative (semi-monotone) operator on $H$, $\tilde{N}$ is a compensated Poisson random measure and $k$ is Lipschitz. We assume the linear growth consition on $f$ and $k$.

We review the related works in the literature. There are many works for the case of Wiener noise, i.e the equation
\[     dX_t=AX_t dt+f(t,X_t) dt + g(t,X_t) dW_t, \]
where $A$ in the generator of a $C_0$-semigroup and $f$ is a non-linear and dissipative operator. This equation has been studied by two method. First is the variational method for which see Krylov and Rosovski~\cite{Krylov-Rozovskii} and also Pardoux~\cite{Pardoux}. Second is the semigroup approach for which see Da Prato and Zabczyk~\cite{DaPrato_Zabczyk_book} and also Zangeneh~\cite{Zangeneh-Paper}. The $L^p$ stability gas been established by Jahanipour and Zangeneh~\cite{Jahanipour-Zangeneh}.

To mention some recent advances for the case of Wiener noise, Barbu~\cite{Barbu-semilinear} has studied this equation in the special case of $H=L^2(D)$ where $D$ is a domain in $\mathbb{R}^d$ and $A$ being the Laplacian and $f$ being a semi-monotone evaluation operator. He does not assume any polynomial growth assumption on $f$. The idea is to reduce the equation to a deterministic equation with random coefficient. Marinelli~\cite{Marinelli-Lp} has studied this equation (with Wiener noise) on the Banach space $L^q(D)$ and $f$ being a semi-monotone evaluation operator. He assumes some polynomial growth on the condition. He uses the stochastic analysis on Banach spaces.

The case of Poisson noise has been less studied. We mention some works that are most related to our work. For more, see the monograph~\cite{Peszat-Zabczyk} and the references therein. There are some works considering general noise (including Poisson noise) and $L^2$ theory, see for example~\cite{Kotelenez-1982} for the Lipschitz coefficients and and~\cite{Gyongy} for the variational approach to the monotone coefficients. They have studied the well-posedness of solutions. See also~\cite{Peszat-Zabczyk} for the monotone coefficients and the additive noise. The semigroup approach for the monotone coefficients and Poisson noise has been studied by Salavati and Zangeneh~\cite{Salavati_Zangeneh:existence and uniqueness} and~\cite{Salavati-Zangeneh-continuity} and the $L^2$-well-posedness has been established. A recent work considering $L^p$ theory is~\cite{Marinelli-Rockner-wellposedness} which proves the well-posedness and asymptotic behaviour for solutions of~\eqref{Poisson_Noise} on certain $L^q(D)$ spaces where $f$ is the evaluation operator associated with a decreasing function.

Our main contribution is to prove the $L^p$-well-posedness of the mild solution of~\eqref{Poisson_Noise} in general Hilbert space $H$ which to the best of our knowledge, has not been proved before in this generality. To see some concrete applications of these results to stochastic partial differential equations and stochastic delay equations see~\cite{Salavati-Zangeneh-continuity}.

The rest of the text is organized as follows. In section~\ref{sec:convolution} we state some tools about the stochastic convolution integrals which we need in our proofs. The main tool is the Ito type inequality for $p$th power which gives a pathwise bound for the $p$th power of the norm of stochastic convolution integrals. In section~\ref{sec: Semilinear SEE} we prove our main results. We prove the existence of the mild solution in $L^p$ in Theorem~\ref{theorem: existence in L^p}. The precise assumptions on coefficients will be stated in this section. We also prove an auxiliary result known as the Bichteler-Jacod inequality in Hilbert spaces proved in Theorem~\ref{theorem: maximal inequality Poisson}. This result has been stated and proved before in the literature, for example in~\cite{Marinelli-Prevot-Rockner}, but we give a new proof for it. We also show the continuous dependence on initial condition and provide a sufficient condition for exponential stability in $L^p$ for the mild solution in Theorem~\ref{theorem: stability}.

\section{Stochastic Convolution Integrals}\label{sec:convolution}

In this section, we review some preliminaries on stochastic convolution integrals and introduce the main tools for the next section.

Stochastic convolution integrals are the solutions of the simplest stochastic evolution equations, i.e. the linear evolution equations with additive noise,
\[ dX_t = AX_t dt + dM_t \]

Let $H$ be a separable Hilbert space with inner product $\langle \, , \, \rangle$. Let $S_t$ be a $C_0$ semigroup on $H$ with infinitesimal generator $A:D(A)\to H$. Furthermore we assume the exponential growth condition on $S_t$, i.e. there exists a constant $\alpha$ such that $\| S_t \| \le e^{\alpha t}$. If $\alpha=0$, $S_t$ is called a contraction semigroup. Stochastic convolution integrals are integrals of type $X_t=\int_0^t S_{t-s} dM_s$ where $M_t$ is a martingale.

The reason for importance of stochastic convolution integrals is that the stochastic evolution equation,
\[ dX_t = AX_t dt + f(X_t) dt + g(X_t) dM_t \]
is written in the following convolution integral form, in order to have solution,
\[ X_t = S_t X_0 + \int_0^t S_{t-s} f(X_s) ds +\int_0^t S_{t-s} g(X_s) dM_s \]
The solutions thus performed are called mild solutions. In this article we are concerned mainly with the well-posedness of mild solution.

Stochastic convolution integrals in Hilbert spaces are different from stochastic integrals in some ways. For example, they are not semi-martingales and hence the tools which have been developed for stochastic integrals are not applicable to them. These include the regularity properties and maximal inequalities.

Kotelenez~\cite{Kotelenez-1982} and~\cite{Kotelenez-1984} gives a maximal inequality for stochastic convolution integrals and also proves the existence of a c\`adl\`ag modification of them. The same estimates for the expectation of the maximum norm of stochastic convolution integrals are given independently by Ichikawa~\cite{Ichikawa} but because of the presence of monotone nonlinearity in our equation, they are not applicable and we need a pathwise bound for stochastic convolution integrals. For this reason we state the following pathwise inequality for the norm of stochastic convolution integrals which has been proved in Salavati and Zangeneh~\cite{Salavati-Zangeneh-pth_moment}.

\begin{theorem}[It\^o type Inequality for $p$th power]\label{theorem: Ito type inequality for pth power}
    Let $p\ge 2$. Assume $Z(t)=V(t) + M(t)$ is a semimartingale where $V(t)$ is an $H$-valued process with finite variation $|V|(t)$ and $M(t)$ is an $H$-valued square integrable martingale with quadratic variation $[M](t)$. Assume that
        \[ \mathbb{E} [M](T)^{\frac{p}{2}} <\infty \qquad \qquad \mathbb{E}|V|(T)^p <\infty \]
    Let $X_0(\omega)$ be $\mathcal{F}_0$ measurable and square integrable. Define $X(t)=S(t) X_0 + \int_0^t S(t-s) dZ(s)$. Then we have
    \begin{multline*}
        \|X(t)\|^p \le e^{p\alpha t}\|X_0\|^p + p\int_0^t e^{p\alpha (t-s)} \|X(s^-)\|^{p-2} \langle X(s^-) , dZ(s) \rangle \\
         + \frac{1}{2}p(p-1)\int_0^t e^{p\alpha (t-s)} \|X(s^-)\|^{p-2} d[M]^c(s)\\
         + \sum_{0\le s\le t} e^{p\alpha (t-s)} \left( \|X(s)\|^p -\|X(s^-)\|^p - p \|X(s^-)\|^{p-2} \langle X(s^-) , \Delta Z(s) \rangle \right)
    \end{multline*}
\end{theorem}

We will need also the following inequality which is an analogous of Burkholder-Davies-Gundy inequality for stochastic convolution integrals.

\begin{theorem}[Burkholder Type Inequality, Zangeneh~\cite{Zangeneh-Paper}, Theorem 2, page 147]\label{theorem:Burkholder type inequality}
    Let $p\ge 2$ and $T>0$. Let $S_t$ be a contraction semigroup on $H$ and $M_t$ be an $H$-valued square integrable c\`adl\`ag martingale for $t\in[0,T]$. Then
    \begin{equation*}
        \mathbb{E}\sup\limits_{0 \le t \le T}\|{\int_0^t S_{t-s}dM_s}\|^p \le K_p \mathbb{E}([M]_T^\frac{p}{2})
    \end{equation*}
    where $K_p$ is a constant depending only on $p$.
\end{theorem}

In the next section we will need an estimate for the $L^p$ norm of stochastic integrals with respect to compensated Poisson random measures. For this reason we state and prove the following theorem which is a Bichteler-Jacod inequality for Poisson integrals in infinite dimensions. This theorem is essentially the Lemma 4 of~\cite{Marinelli-Rockner-wellposedness} with an extension to $1\le p\le 2$. We provide a new proof for this theorem based on the Burkholder-Davies -Gundy inequality.

\begin{theorem}[An $L^p$ bound for Stochastic Integrals with Respect to Compensated Poisson Random Measures]\label{theorem: L^p_bound_stochastic_integral}
    Let $p\ge 1$. There exists a real constant denoted by $\mathfrak{C}_p$ such that if $k(t,\xi,\omega)$ is an $H$-valued predictable process then
    \begin{equation} \label{theorem: maximal inequality Poisson}
    \begin{array}{l}
        \mathbb{E} \sup_{0\le t\le T} \left|\int_0^t\int_E k(s,\xi,\omega) \tilde{N}(ds,d\xi)\right|^p \le \\
        \mathfrak{C}_p \left( \mathbb{E} \left( \left(\int_0^T \int_E \left|k(s,\xi,\omega)\right| \nu(d\xi)ds\right)^{p} \right) + \mathbb{E} \int_0^T \int_E \left|k(s,\xi,\omega)\right|^p \nu(d\xi)ds\right)
    \end{array}
    \end{equation}
\end{theorem}

\begin{proof}
    Assume that $2^n\le p < 2^{n+1}$. We prove by induction on $n$.

    \emph{Basis of induction:} $n=0$. In this case we have $1\le p<2$ and the statement follows from Theorem 8.23 of~\cite{Peszat-Zabczyk}. In fact, in this case we have 
    \[\mathbb{E} \left|\int_0^t\int_E k(s,\xi,\omega) \tilde{N}(ds,d\xi)\right|^p \le 
        \mathfrak{C}_p \mathbb{E} \int_0^t \int_E \left|k(s,\xi,\omega)\right|^p \nu(d\xi)ds\]
    \emph{Induction Step:} Now assume $n\ge 1$ and we have proved the statement for $n-1$. Hence $p\ge 2$. Applying Burkholder-Davies-Gundy inequality we find
    \begin{multline*}
    \mathbb{E} \sup_{0\le t\le T} \left|\int_0^t\int_E k(s,\xi,\omega) \tilde{N}(ds,d\xi)\right|^p \le \\
     K_p \mathbb{E} \left( (\int_0^T \int_E \|k(s,\xi,\omega)\|^2 N(ds,d\xi))^\frac{p}{2}\right)
    \end{multline*}
	Subtracting a compensator from the right hand side we get
    \begin{multline*}
        \le K_p 2^{\frac{p}{2}} \Bigg[ \mathbb{E} \left( \left(\int_0^T \int_E \left|k(s,\xi,\omega)\right|^2 \tilde{N}(ds,d\xi) \right)^\frac{p}{2}\right) \\
        + \mathbb{E} \left( \left(\int_0^T \int_E \left|k(s,\xi,\omega)\right|^2 \nu(d\xi)ds \right)^\frac{p}{2}\right) \Bigg]
    \end{multline*}

    Note that $2^{n-1}\le \frac{p}{2} <2^n$ hence we can apply the induction hypothesis to the first term on the right hand side and find
    \begin{multline}\label{equation: proof of L^p_bound_1}
        \le K_p 2^{\frac{p}{2}} \Bigg[ \mathfrak{C}_{\frac{p}{2}} \mathbb{E} \left( \left( \int_0^T \int_E \left|k(s,\xi,\omega)\right|^2 \nu(d\xi) ds \right)^{\frac{p}{2}}\right)\\
        + \mathbb{E} \left( \int_0^T \int_E \left|k(s,\xi,\omega)\right|^p \nu(d\xi) ds \right) \\
        + \mathbb{E} \left( \left(\int_0^T \int_E \left|k(s,\xi,\omega)\right|^2 \nu(d\xi)ds \right)^\frac{p}{2}\right)\Bigg]
    \end{multline}
    By the interpolation inequality for a suitable $\theta$ such that $\theta+\frac{1-\theta}{p}=\frac{1}{2}$ we have
    \begin{multline*}
    \left( \int_0^T \int_E \left|k(s,\xi,\omega)\right|^2 \nu(d\xi) ds \right)^{\frac{1}{2}} \le \\
    \left( \int_0^T \int_E \left|k(s,\xi,\omega)\right| \nu(d\xi) ds \right)^{\theta} \left( \int_0^T \int_E \left|k(s,\xi,\omega)\right|^p \nu(d\xi) ds \right)^{\frac{1-\theta}{p}}
    \end{multline*}
    raising to power $p$ we have
    \begin{multline*}
    \left( \int_0^T \int_E \left|k(s,\xi,\omega)\right|^2 \nu(d\xi) ds \right)^{\frac{p}{2}} \le \\
    \left( \int_0^T \int_E \left|k(s,\xi,\omega)\right| \nu(d\xi) ds \right)^{\theta p} \left( \int_0^T \int_E \left|k(s,\xi,\omega)\right|^p \nu(d\xi) ds \right)^{1-\theta}
    \end{multline*}
    And by the arithmetic-geometric mean inequality
    \[ \le \theta \left( \int_0^T \int_E \left|k(s,\xi,\omega)\right| \nu(d\xi) ds \right)^{{p}} + (1-\theta)\left( \int_0^T \int_E \left|k(s,\xi,\omega)\right|^p \nu(d\xi) ds \right) \]
    taking expectations and substituting in~\eqref{equation: proof of L^p_bound_1} the statement is proved.
\end{proof}

\section{Semilinear Stochastic Evolution Equations with L\'evy Noise and Monotone Nonlinear Drift}\label{sec: Semilinear SEE}

Let $(\Omega,\mathcal{F},\mathcal{F}_t,\mathbb{P})$ be a filtered probability space. Let $(E,\mathcal{E})$ be a measurable space and $N(dt,d\xi)$ a Poisson random measure on $\mathbb{R}^+ \times E$ with intensity measure $dt \nu(d\xi)$. Our goal is to study the following equation in $H$,
\begin{equation}\label{main_equation}
    dX_t=AX_t dt+f(t,X_t) dt + \int_E k(t,\xi,X_{t-}) \tilde{N}(dt,d\xi),
\end{equation}
where $\tilde{N}(dt,d\xi)=N(dt,d\xi)-dt\nu(d\xi)$ is the compensated Poisson random measure corresponding to $N$.

The existence and uniqueness of the mild solutions of these equations in $L^2$ has been proved in~\cite{Salavati_Zangeneh:existence and uniqueness}. In this section we prove the existence and uniqueness of the solution in $L^p$ for $p\ge 2$ in Theorem~\ref{theorem: existence in L^p}. We also prove the continuous dependence on initial condition and provide sufficient conditions under which the solutions are exponentially asymptotically stable.

We will use the notion of stochastic integration with respect to compensated Poisson random measure. For the definition and properties see~\cite{Peszat-Zabczyk} and~\cite{Albeverio-Mandrekar-Rudiger-2009}.

\begin{definition}
    $f:H\to H$ is called \emph{demicontinuous} if whenever $x_n \to x$, strongly in $H$ then $f(x_n)\rightharpoonup f(x)$ weakly in $H$.
\end{definition}

We assume the following,

\begin{hypothesis}\label{main_hypothesis}
    \begin{description}

        \item[(a)] $f(t,x,\omega):\mathbb{R}^+\times H\times \Omega \to H$ is measurable, $\mathcal{F}_t$-adapted, demicontinuous with respect to $x$ and there exists a constant $M$ such that
            \[ \langle f(t,x,\omega)-f(t,y,\omega),x-y \rangle \le M \|x-y\|^2,\]

        \item[(b)] $k(t,\xi,x,\omega):\mathbb{R}^+\times E\times H\times \Omega \to H$ is predictable and there exists a constant $C$ such that
            \[ \int_{E}\|k(t,\xi,x)-k(t,\xi,y)\|^2 \nu(d\xi) \le C \|x-y \|^2,\]

        \item[(c)] There exists a constant $D$ such that
            \[ \| f(t,x,\omega)\|^2 + \int_{E}\|k(t,\xi,x)\|^2 \nu(d\xi) \le D(1+\|x\|^2),\]
		
		\item[(d)] There exists a constant $F$ such that
            \[ \int_{E}\|k(t,\xi,x)-k(t,\xi,y)\|^p \nu(d\xi) \le F \|x-y \|^p,\]
            \[\int_{E}\|k(t,\xi,x)\|^p \nu(d\xi) \le F(1+\|x\|^p),\]
            
        \item[(e)] $X_0(\omega)$ is $\mathcal{F}_0$ measurable and $\mathbb{E}\|X_0\|^p < \infty$.
    \end{description}

\end{hypothesis}

\begin{definition}
    By a \emph{mild solution} of equation~\eqref{main_equation} with initial condition $X_0$ we mean an adapted c\`adl\`ag process $X_t$ that satisfies
    \begin{multline}\label{mild_solution}
        X_t=S_t X_0+\int_0^t S_{t-s}f(s,X_s) ds + \int_0^t{\int_E {S_{t-s}k(s,\xi,X_{s-})} \tilde{N}(ds,d\xi).}
    \end{multline}
\end{definition}

We are naw ready to state and prove the main theorem of this section,

\begin{theorem}[Existence of the Solution in $L^p$]\label{theorem: existence in L^p}
Let $p\ge 2$. Then under assumptions of Hypothesis~\ref{main_hypothesis}, equation~\eqref{main_equation} has a unique square integrable c\`adl\`ag mild solution $X(t)$ such that $\mathbb{E} \sup_{0\le s\le t} \|X(s)\|^p <\infty$.
\end{theorem}

Before proceeding to proof of Theorem~\ref{theorem: existence in L^p}, we state the following theorem from~\cite{Zangeneh-measurability} without proof.

\begin{theorem}[Zangeneh,~\cite{Zangeneh-measurability}]\label{theorem:measurability}
Assume $f$ satisfies Hypothesis~\ref{main_hypothesis}-(a) and there exists a constant $D$ such that $\|f(t,x,\omega)\|^2 \le D (1+\|x\|^2)$ and assume $V(t,\omega)$ is an adapted process with c\`adl\`ag trajectories and $X_0(\omega)$ is $\mathcal{F}_0$ measurable. Then the equation,
        \[ X_t=S_t X_0 + \int_0^t S_{t-s}f(s,X_s,\omega) ds + V(t,\omega)\]
    has a unique measurable adapted c\`adl\`ag solution $X_t(\omega)$. Furtheremore
    \[	\|X(t)\|\le \|X_0\|+\|V(t)\|+\int_0^t e^{(\alpha+M)(t-s)} \|f(s,S_s X_0+V(s))\| ds, \]
\end{theorem}

\begin{lemma}\label{lemma: alpha=0}
        It suffices to prove theorem~\ref{theorem: existence in L^p} for the case that $\alpha=0$.
    \end{lemma}
    \begin{proof} Define
        \begin{gather*}
           \tilde{S}_t= e^{-\alpha t} S_t ,\qquad \tilde{f}(t,x,\omega)=e^{-\alpha t}f(t,e^{\alpha t}x,\omega) , \\
           \tilde{k}(t,\xi,x,\omega)=e^{-\alpha t}k(t,\xi,e^{\alpha t}x,\omega).
        \end{gather*}
        Note that $\tilde{S}_t$ is a contraction semigroup. It is easy to see that $X_t$ is a mild solution of equation~\eqref{main_equation} if and only if $\tilde{X}_t=e^{-\alpha t} X_t$ is a mild solution of equation with coefficients $\tilde{S},\tilde{f},\tilde{k}$.
    \end{proof}

\begin{proof}[Proof of Theorem~\ref{theorem: existence in L^p}.]
	Existence and uniqueness of the mild solution in $L^2$ has been proved in~\cite{Salavati_Zangeneh:existence and uniqueness}, Theorem 4. Uniqueness in $L^2$ implies the uniqueness in $L^p$ for $p\ge 2$. It remains to prove the existence in $L^p$.

    \emph{Existence.}
	It suffices to prove the existence of a solution on a finite interval $[0,T]$. Then one can show easily that these solutions are consistent and give a global solution. We define adapted c\`adl\`ag processes $X^{n}_t$ recursively as follows. Let $X^0_t=S_t X_0$. Assume $X^{n-1}_t$ is defined. Theorem~\ref{theorem:measurability} implies that there exists an adapted c\`adl\`ag solution $X^n_t$ of
    \begin{equation}\label{equation: proof of existence_iteration}
        X^n_t=S_t X_0 + \int_0^t S_{t-s}f(s,X^n_s) ds + V^n_t,
    \end{equation}
    where
        \[ V^{n}_t= \int_0^t{\int_E {S_{t-s}k(s,\xi,X^{n-1}_{s-})} \tilde{N}(ds,d\xi)}. \]
    
	It is proved in~\cite{Salavati_Zangeneh:existence and uniqueness} that $\{X^n\}$ converge to some adapted c\`adl\`ag process $X_t$ in the sense that 
	\[ \mathbb{E} \sup\limits_{0\le t\le T} \|X^n_t-X_t\|^2 \to 0, \]
	and that $X_t$ is the mild solution of equation~\eqref{main_equation}.
	
    We wish to show that $\{X^n\}$ converge to $X_t$ in $L^p$ with the supremum norm. This is done by the following two lemmas.
    
    \begin{lemma}\label{lemma: finite_pth_moment_iteration}
    	\[ \mathbb{E}\sup\limits_{0\le t\le T} \|X^n_t\|^p<\infty. \]
    \end{lemma}

	\begin{proof}
		We prove by induction on $n$. By Theorem~\ref{theorem:measurability} we have the following estimate,
	    \[	\|X^n_t\|\le \|X_0\|+\|V^n_t\|+\int_0^t e^{M(t-s)} \|f(s,S_s X_0+V^n_s)\| ds. \]
	  	Hence,
	   \[ \|X^n_t\|^p \le 3^p \|X_0\|^p + 3^p \|V^n_t\|^p + 3^p \left( \int_0^t e^{M(t-s)} \|f(s , S_s X_0+V^n_s)\| ds \right)^p \]
	  	Taking supremum and using Cauchy-Schwartz inequality we find	
	  \begin{multline*}
	  		\sup\limits_{0\le t\le T} \|X^n_t\|^p \le 3^p \|X_0\|^p + 3^p \sup\limits_{0\le t\le T} \|V^n_t\|^p\\
	  		+ 3^p e^{|M|T} T^{\frac{p}{2}} \left( \underbrace{\int_0^T \|f(s , S_s X_0+V^n_s)\|^2 ds}_G \right)^\frac{p}{2}
	  \end{multline*}
	  	Using Hypothesis~\ref{main_hypothesis}-(c) and Holder's inequality we find
	  	\begin{multline*}
	   		G \le D^\frac{p}{2} \left( \int_0^T (1+\|S_s X_0 +V^n_s\|^2) ds \right)^\frac{p}{2} \\
	   			\le D^\frac{p}{2} \left( T + 2T \|X_0\|^2 + 2 \int_0^T \|V^n_s\|^2 ds \right) ^\frac{p}{2} \\
	   			\le D^\frac{p}{2} \left( 3^\frac{p}{2} T^\frac{p}{2} + 2^\frac{p}{2}3^\frac{p}{2} T^\frac{p}{2} \|X_0\|^p + 2^\frac{p}{2}3^\frac{p}{2} T^\frac{p}{2} \sup\limits_{0\le s\le T} \|V^n_s\|^2 \right)
	   	\end{multline*}
	   	
	   	Hence, to prove the Lemma it suffices to prove that 
	    \[ \mathbb{E}\sup\limits_{0\le t\le T} \|V^n_t\|^p<\infty. \]

	    Applying Burkholder type inequality (Theorem~\ref{theorem:Burkholder type inequality}), we find
	    \[ \mathbb{E}\sup\limits_{0\le t\le T} \|V_t\|^p \le K_p \mathbb{E} ([\tilde{M}]_T^{\frac{p}{2}}), \]
	    where $\tilde{M}_t=\int_0^t \int_E k(s,\xi,X^{n-1}_{s-}) \tilde{N}(ds,du)$. Hence
	    \begin{multline*}
	    		\mathbb{E}\sup\limits_{0\le t\le T} \|V_t\|^p \le K_p \mathbb{E} \left( (\int_0^T \int_E \|k(s,\xi,X^{n-1}_{s})\|^2 N(ds,du))^\frac{p}{2}\right) \\
	    		\le 2^\frac{p}{2} K_p \Bigg( \mathbb{E} \left( (\int_0^T \int_E \|k(s,\xi,X^{n-1}_{s})\|^2 \nu(du) ds)^\frac{p}{2} \right) \\
	    	+ \mathbb{E} \left( (\int_0^T \int_E \|k(s,\xi,X^{n-1}_{s})\|^2 \tilde{N}(ds,du))^\frac{p}{2}\right) \Bigg)
	    \end{multline*}
	    By Hypothesis~\ref{main_hypothesis} (c) we have,
	    \begin{multline*}
	    		\le 2^\frac{p}{2} K_p D^\frac{p}{2} \left(\mathbb{E}( \int_0^T (1+\|X^{n-1}_s\|^2) ds)^\frac{p}{2})\right) \\
	    		+ 2^\frac{p}{2} K_p \mathbb{E} \left( (\int_0^T \int_E \|k(s,\xi,X^{n-1}_{s})\|^2 \tilde{N}(ds,du))^\frac{p}{2}\right)
	    \end{multline*}
	    Since $\frac{p}{2}\ge 1$, we can apply Theorem~\ref{theorem: L^p_bound_stochastic_integral} to second term and find
	    \begin{multline}\label{equation: proof of finite pth moment-1}
	    		\le 2^\frac{p}{2} K_p D^\frac{p}{2} \left(\mathbb{E}( \int_0^T (1+\|X^{n-1}_s\|^2) ds)^\frac{p}{2})\right) \\
	    		+ 2^\frac{p}{2} K_p \mathfrak{C}_\frac{p}{2} \left(\mathbb{E}( \int_0^T (1+\|X^{n-1}_s\|^2) ds)^\frac{p}{2})\right) \\
	    		+ 2^\frac{p}{2} K_p \mathfrak{C}_\frac{p}{2} \mathbb{E} \int_0^T (1+\|X^{n-1}_s\|^p) ds) \\
	    \end{multline}
	    By Hypothesis~\ref{main_hypothesis} (c) we find,
	    \begin{multline}\label{equation: proof of finite pth moment-4}
	    		\le 2^\frac{p}{2} K_p \left((D\int_0^T \mathbb{E}\|X^{n-1}_s\|^2 ds)^\frac{p}{2}\right. \\
	    		\left. + \mathfrak{C}_\frac{p}{2} \left( \left( D\int_0^T \mathbb{E}\|X^{n-1}_s\|^2 ds \right)^\frac{p}{2} + D \left( \int_0^T \mathbb{E}\|X^{n-1}_s\|^p ds \right) \right) \right)\\
	    		  \le C_1 \left( (\int_0^T \mathbb{E}\|X^{n-1}_s\|^2 ds)^\frac{p}{2} \right) + C_2 \left( \int_0^T \mathbb{E}\|X^{n-1}_s\|^p ds \right) 
	    \end{multline}
	    where $C_1=2^\frac{p}{2} K_p D (1+\mathfrak{C}_\frac{p}{2})$ and $C_2=2^\frac{p}{2} K_p \mathfrak{C}_\frac{p}{2} D$, now by Holder inequality we find,
	    \begin{multline}\label{equation: proof of finite pth moment-5}
	    		\le C_3 \left( \int_0^T \mathbb{E}\|X^{n-1}_s\|^p ds \right) 
	    \end{multline}
	    which is finite by induction. The basis of induction follows directly from Hypothesis~\ref{main_hypothesis}-(e).
	\end{proof}

    \begin{lemma}\label{lemma: convergence_pth_moment}
        For $0\le t\le T$,
        \begin{equation}\label{equation: proof of existence_supremum convergence}
            \mathbb{E} \|X^{n+1}_t-X^n_t\|^p \le C_0 C_1^n \frac{t^n}{n!}
        \end{equation}
        where $C_0$ and $C_1$ are constants that are introduced below.
    \end{lemma}

    \begin{proof}
        We prove by induction on $n$. Assume that the statement is proved for $n-1$. We have,
        \begin{equation}\label{equation: proof of existence_X^n+1-X^n}
            X^{n+1}_t-X^n_t= \int_0^t S_{t-s}(f(s,X^{n+1}_s)-f(s,X^n_s)) ds + \int_0^t S_{t-s} dM_s,
        \end{equation}
        where
        \begin{eqnarray*}
            M_t &=&\int_0^t \int_E {(k(s,\xi,X^{n}_{s-})-k(s,\xi,X^{n-1}_{s-}))\tilde{N}(ds,d\xi)}.
        \end{eqnarray*}
        Applying Theorem~\ref{theorem: Ito type inequality for pth power}, for $\alpha=0$, we have
        \begin{multline}\label{equation:proof of existence 1}
            \lVert X^{n+1}_t-X^n_t \rVert ^p \le \\
            p \underbrace{\int_0^t {\|X^{n+1}_s-X^n_s\|^{p-2} \langle X^{n+1}_{s}-X^n_{s},f(s,X^{n+1}_s)-f(s,X^n_s)\rangle ds}}_{A^n_t}\\
            + p \underbrace{\int_0^t \|X^{n+1}_{s-}-X^n_{s-}\|^{p-2} \langle X^{n+1}_{s-}-X^n_{s-},dM_s\rangle}_{B^n_t}\\
		  + \frac{1}{2}p(p-1)  \underbrace{ \int_0^t \|X^{n+1}_{s-}-X^n_{s-}\|^{p-2} d[M]^c_s}_{C^n_t} + \int_0^t \int_E D^n_s N(ds,d\xi)
        \end{multline}
where
	\begin{multline*}
		D^n_s = \|X^{n+1}_{s-}-X^n_{s-} + k(s,\xi,X^n_{s-})-k(s,\xi,X^{n-1}_{s-})\|^p - \|X^{n+1}_{s-}-X^n_{s-}\|^p \\
		- p \|X^{n+1}_{s-}-X^n_{s-}\|^{p-2} \langle X^{n+1}_{s-}-X^n_{s-},k(s,\xi,X^n_{s-})-k(s,\xi,X^{n-1}_{s-})\rangle
	\end{multline*}
        Note that for a c\`adl\`ag function the set of discontinuity points is countable, hence when integrating with respect to Lebesgue measure, they can be neglected. Therefore from now on, we neglect the left limits in integrals with respect to Lebesgue measure. So, for the term $A_t$, the semimonotonicity assumption on $f$ implies
        \begin{equation}\label{equation:proof of existence_A_t}
            A^n_t \le M \int_0^t \|X^{n+1}_s-X^n_s\|^p ds
        \end{equation}
        We also have
        \[ [M]^c_t =0 \]
        and hence
        \begin{equation*}
           C^n_t =0
        \end{equation*}
		For the term $D^n_s$ we have by Lemma~\ref{lemma: inequality in proof of Ito-type pth power} (proved later),
		\begin{multline*}
			D^n_s \le \frac{1}{2} p (p-1) \Bigg(  \|X^{n+1}_{s-}-X^n_{s-}\|^{p-2} \\
			+ \|k(s,\xi,X^n_{s-})-k(s,\xi,X^{n-1}_{s-})\|^{p-2} \Bigg) \|k(s,\xi,X^n_{s-})-k(s,\xi,X^{n-1}_{s-})\|^2
		\end{multline*}
		Hence by Hypothesis~\ref{main_hypothesis}-(b) and (d),
		\begin{multline}\label{equation:proof of existence_E(int D_s)}
        	\mathbb{E} \int_E D^n_s \nu(d\xi) \le \frac{1}{2} p (p-1) \Bigg(  C \mathbb{E} \left( \|X^{n+1}_{s-}-X^n_{s-}\|^{p-2}  \|X^{n}_s-X^{n-1}_s\|^2 \right) \\
        	+ F  \mathbb{E} \left( \|X^{n}_{s-}-X^{n-1}_{s-}\|^{p}\right) \Bigg)
        \end{multline}

        Now, taking expectations on both sides of~\eqref{equation:proof of existence 1} and substituting~\eqref{equation:proof of existence_A_t} and~\eqref{equation:proof of existence_E(int D_s)} and noting that $B_t$ is a martingale we find,
        \begin{multline*}
        	\mathbb{E} \|X^{n+1}_t-X^n_t\|^p \le p M \int_0^t \mathbb{E}\|X^{n+1}_s-X^n_s\|^p ds \\
        	+ \frac{1}{2} p(p-1)C \int_0^t \mathbb{E}\left(\|X^{n+1}_s-X^n_s\|^{p-2}\|X^n_s-X^{n-1}_s\|^2\right)ds\\
        	+ \frac{1}{2} p (p-1) F \mathbb{E} \|X^n_s-X^{n-1}_s\|^p ds.
        \end{multline*}
        Applying Holder's inequality to the second integral in the right hand side we find
        \begin{multline*}
        	\le p M \int_0^t \mathbb{E}\|X^{n+1}_s-X^n_s\|^p ds \\
        	+ \frac{1}{2} p(p-1)C \left( \frac{p-2}{p}\int_0^t \mathbb{E}\|X^{n+1}_s-X^n_s\|^ ds + \frac{2}{p}\int_0^t \mathbb{E}\|X^n_s-X^{n-1}_s\|^p ds \right)\\
        	+ \frac{1}{2} p (p-1) F \mathbb{E} \|X^n_s-X^{n-1}_s\|^p ds\\
        	\le \beta  \int_0^t \mathbb{E}\|X^{n+1}_s-X^n_s\|^p ds + \gamma \int_0^t \mathbb{E}\|X^n_s-X^{n-1}_s\|^p ds
        \end{multline*}
        where $\beta= pM+ \frac{1}{2} (p-1)(p-2) C$ and $\gamma=\frac{1}{2} (p-1) (2C + pF)$.

	Define $h^n(t)=\mathbb{E}\|X^{n+1}_t-X^n_t\|^p$. We have
			\[ h^n(t) \le \beta \int_0^t h^n(s) ds + \gamma \int_0^t h^{n-1}(s) ds \]
       By Lemma~\ref{lemma: finite_pth_moment_iteration} we know that $h^n(t)$ is uniformly bounded with respect to $t$, hence we can use Gronwall's inequality and find
       	\begin{equation*}
            h^n(t)\le \gamma e^{\beta t} \int_0^t h^{n-1}(s) ds
        \end{equation*}
        We have $h^0(t)\le C_0$ where $C_0=2^p \mathbb{E} \sup_{0\le t \le T} (\|X^1_t\|^p + \|X^0_t\|^p)<\infty$ and it follows inductively that,
        	\[ h^n(t) \le C_0 C_1^n \frac{t^n}{n!} \]
        where $C_1=\gamma e^{\beta T}$.
       \end{proof}

    Back to the proof of Theorem~\ref{theorem: existence in L^p}, since the right hand side of~\eqref{equation: proof of existence_supremum convergence} is a convergent series, $\{X^n\}$ is a cauchy sequence in $L^p(\Omega,\mathcal{F},\mathbb{P};L^\infty([0,T];H))$ and hence converges to a process $Y_t(\omega)$. But as is proved in \cite{Salavati_Zangeneh:existence and uniqueness}, $\{X^n\}$ converges to a process $X_t$ in $L^2(\Omega,\mathcal{F},\mathbb{P};L^\infty([0,T];H))$ which is a solution of equation~\eqref{main_equation}. Hence $Y_t=X_t$.
\end{proof}

In the above proof, we have used the following lemma,
\begin{lemma}\label{lemma: inequality in proof of Ito-type pth power}
            For $x,y\in H$ we have
                \[ \|x+y\|^p-\|x\|^p -p\|x\|^{p-2}\langle x , y \rangle \le \frac{1}{2}p(p-1) (\|x\|^{p-2} +\|x+y\|^{p-2})\|y\|^2 \]
\end{lemma}
\begin{proof}
    Define $f(t)= \|x+ty\|^p$. Then
        \[ f'(t)= p\|x+ty\|^{p-2}\langle x+ty , y \rangle \]
    and
    \begin{multline*}
    f''(t)= p\|x+ty\|^{p-2}\|y\|^2 +p(p-2)\|x+ty\|^{p-4} \langle x+ty , y \rangle ^2 \le \\
    p(p-1)\|x+ty\|^{p-2}\|y\|^2
    \end{multline*}
    By Taylor's remainder theorem we have for some $\tau\in [0,1]$,
        \[ f(1)-f(0)-f'(0)=\frac{1}{2}f''(\tau) \le \frac{1}{2} p(p-1)\|x+\tau y\|^{p-2}\|y\|^2 \]
    But $\|x+\tau y\|\le \max(\|x\|,\|x+y\|)$. Hence
        \[ f(1)-f(0)-f'(0) \le \frac{1}{2}p(p-1) (\|x\|^{p-2} +\|x+y\|^{p-2})\|y\|^2 \]
    which completes the proof.
\end{proof}

\begin{theorem} [Continuous dependence and Exponential Stability in the $p$th Moment]\label{theorem: stability}
    Let $X_t$ and $Y_t$ be mild solutions of~\eqref{main_equation} with initial conditions $X_0$ and $Y_0$. Then
    \begin{eqnarray*}
        \mathbb{E} \| X_t-Y_t \| ^p &\le&  e^{\gamma t} \mathbb{E}\|X_0-Y_0\|^p
    \end{eqnarray*}
    where $\gamma= p \alpha + p M+\frac{1}{2}p(p-1) C+ \frac{1}{2} p (p-1) ((2^{p-2}+1) C + 2^{p-2} F)$. Hence the mild solution is $L^p$ continuous with respect to initial conditions. Moreover, if $\gamma < 0$ then all mild solutions are exponentially stable in the $L^p$ norm.
\end{theorem}

\begin{proof}
	    First we consider the case that $\alpha=0$. Subtract $X_t$ and $Y_t$,
    \begin{multline*}
        X_t-Y_t=S_t (X_0-Y_0)
        + \int_0^t S_{t-s} (f(s,X_s)-f(s,Y_s))ds
        + \int_0^t S_{t-s} dM_s,
    \end{multline*}
    where
    \[ M_t=\int_E (k(s,\xi,X_{s-})-k(s,\xi,Y_{s-}))d\tilde{N}. \]
    Applying It\^o type Inequality for $p$th power (Theorem~\ref{theorem: Ito type inequality for pth power}), for $\alpha=0$, to $X_t-Y_t$ and rewriting it with respect to random Poisson measure, we find
    \begin{multline}\label{equation:proof of stability 1}
        \| X_t-Y_t \| ^p \le \| X_0-Y_0 \| ^p \\
        +  p \underbrace {\int_0^t {\|X_s-Y_s\|^{p-2} \langle X_{s-}-Y_{s-} , (f(s,X_s)-f(s,Y_s))\rangle ds}}_\mathbf{A_t}\\
        + p \underbrace {\int_0^t {\|X_s-Y_s\|^{p-2} \langle X_{s-}-Y_{s-} , d M_s \rangle}}_\mathbf{B_t} \\
        +\frac{1}{2}p(p-1)\underbrace{\int_0^t \|X_s-Y_s\|^{p-2}d[M]_s}_\mathbf{C_t}        + \int_0^t \int_E  \mathbf{D_s} N(ds,d\xi)
    \end{multline}
where
    \begin{multline*}
        \mathbf{D_s}=\|X_{s-}-Y_{s-}+k(s,\xi,X_{s-})-k(s,\xi,Y_{s-})\|^p-\|X_{s-}-Y_{s-}\|^p \\
        -p\|X_{s-}-Y_{s-}\|^{p-2} \langle X_{s-}-Y_{s-},k(s,\xi,X_{s-})-k(s,\xi,Y_{s-}) \rangle .
    \end{multline*}
	
	Using Hypothesis~\ref{main_hypothesis} (a) for term $\mathbf{A}_t$ we find
	\begin{equation}\label{equation:proof of stability 2}
		\mathbb{E}\mathbf{A}_t \le M \int_0^t \mathbb{E} \|X_s-Y_s\|^p ds
	\end{equation}
	Using Hypothesis~\ref{main_hypothesis} (b) for term $\mathbf{C}_t$ we find
	\begin{equation}\label{equation:proof of stability 3}
		\mathbb{E}\mathbf{C}_t \le C \int_0^t \mathbb{E} \|X_s-Y_s\|^p ds
	\end{equation}
	
	For term $\mathbf{D}_s$ we have by Lemma~\ref{lemma: inequality in proof of Ito-type pth power},
	\begin{multline*}
		\mathbf{D}_s \le \frac{1}{2} p (p-1) \Big( \|X_{s-}-Y_{s-}\|^{p-2}
		+ \|X_{s-}-Y_{s-}+k(s,\xi,X_{s-})-k(s,\xi,Y_{s-})\|^{p-2}\Big)\\
		\|k(s,\xi,X_{s-})-k(s,\xi,Y_{s-})\|^2\\
		\le \frac{1}{2} p (p-1) \Big((2^{p-2}+1)\|X_{s-}-Y_{s-}\|^{p-2} + 2^{p-2} \|k(s,\xi,X_{s-})-k(s,\xi,Y_{s-})\|^{p-2}\Big)\\
		\|k(s,\xi,X_{s-})-k(s,\xi,Y_{s-})\|^2
	\end{multline*}
	Using Hypothesis~\ref{main_hypothesis} (b) and (d), we find
	\begin{multline}\label{equation:proof of stability 4}
		\mathbb{E} \int_E \mathbf{D}_s \nu(d\xi) ds \le \frac{1}{2} p (p-1) ((2^{p-2}+1) C + 2^{p-2} F)\mathbb{E} \|X_{s-}-Y_{s-}\|^p
	\end{multline}
    Taking expectations on both sides of~\eqref{equation:proof of stability 1} and noting that $\mathbf{B}_t$ is a martingale and substituting~\eqref{equation:proof of stability 2}, ~\eqref{equation:proof of stability 3} and ~\eqref{equation:proof of stability 4} we find
    
    \begin{equation*}
    	\mathbb{E} \|X_t-Y_t\|^p \le \mathbb{E} \|X_0-Y_0\|^p + \gamma \int_0^t \mathbb{E} \|X_s-Y_s\|^p ds \\
    \end{equation*}
    where $\gamma=p M+\frac{1}{2}p(p-1) C+ \frac{1}{2} p (p-1) ((2^{p-2}+1) C + 2^{p-2} F)$.
	Now applying Gronwall's inequality the statement follows. Hence the proof for the case $\alpha=0$ is complete. Now for the general case, apply the change of variables used in Lemma~\ref{lemma: alpha=0}.

\end{proof}

\begin{remark}
	 The results of this section remain valid by adding a Wiener noise term to equation~\eqref{main_equation}. i.e. for the equation
\begin{equation}\label{Wiener_noise}
    dX_t=AX_t dt+f(t,X_t) dt + g(t,X_{t-})d W_t + \int_E k(t,\xi,X_{t-}) \tilde{N}(dt,d\xi),
\end{equation}
	where $W_t$ is a cylindrical Wiener process on a Hilbert space $K$, independent of $N$ and $g(t,x,\omega):\mathbb{R}^+\times H\times \Omega \to L_{HS}(K,H)$ (Space of Hilbert-Schmidt operators from $K$ to $H$) is Lipschitz and has linear growth. The proofs are straightforward generalizations of the proofs of this section.
\end{remark}

\end{document}